\documentclass[12pt,a4paper]{amsart}
\usepackage{amsmath,amsthm,amssymb,amsfonts}
\usepackage[pdftex]{color}
\usepackage[bookmarks=true,hyperindex,pdftex,colorlinks,citecolor=blue,linkcolor=blue,urlcolor=blue]
{hyperref}

\usepackage{soul}

\theoremstyle{definition}
\newtheorem{theorem}{Theorem}

\newtheorem{prop}[theorem]{Proposition}
\newtheorem{lemma}[theorem]{Lemma}

\theoremstyle{definition}

\newcommand{\R}{\mathbb{R}}
\newcommand{\N}{\mathbb{N}}

\newcommand{\C}{\mathbb{C}}
\newcommand{\K}{\mathbb{K}}

\newcommand{\e}{\varepsilon}
\newcommand{\eps}{\varepsilon}

\def\A{\mathcal{A}_{w^*u}(B_{X^*})}

\DeclareMathOperator{\re}{Re}

\DeclareMathOperator{\ext}{Ext}

\renewcommand{\leq}{\leqslant}
\renewcommand{\geq}{\geqslant}
\renewcommand{\le}{\leqslant}
\renewcommand{\geq}{\geqslant}

\title{A non-linear Bishop-Phelps-Bollob\'as type theorem}

\author[S.~Dantas]{Sheldon Dantas}
\address{Departamento de An\'{a}lisis Matem\'{a}tico, Universidad de Valencia, Doctor Moliner 50, 46100 Burjasot (Valencia), Spain} \email{sheldon.dantas@uv.es}

\author[D.~Garc\'ia]{Domingo Garc\'ia}
\address{Departamento de An\'{a}lisis Matem\'{a}tico, Universidad de Valencia, Doctor Moliner 50, 46100 Burjasot	(Valencia), Spain} \email{domingo.garcia@uv.es}

\author[S. K. Kim]{Sun Kwang Kim}
\address{Department of Mathematics, Kyonggi University, 443-760 (Suwon), Republic of Korea}
\email{sunkwang@kgu.ac.kr}

\author[U. Y. Kim]{Un Young Kim}
\address{Department of Mathematics,	POSTECH, 790-784, (Pohang), Republic of Korea}
\email{kuy1219@postech.ac.kr}

\author[H. J. Lee]{Han Ju Lee}
\address{Department of Mathematics Education, Dongguk University, 100-715 (Seoul), Republic of Korea}
\email{hanjulee@dongguk.edu}

\author[M.~Maestre]{Manuel Maestre}
\address{Departamento de An\'{a}lisis Matem\'{a}tico,
Universidad de Valencia, Doctor Moliner 50, 46100 Burjasot	(Valencia), Spain}
\email{manuel.maestre@uv.es}

\thanks{First author was supported by MINECO and FEDER project MTM2014-57838-C2-2-P, and Ci\^encias Sem Fronteiras - Doutorado Pleno grant 0050/13-0. Second and sixth authors was supported by MINECO and FEDER project MTM2014-57838-C2-2-P, and Prometeo II/2013/013. Third author was partially supported by Basic Science Research Program through the National Research Foundation of Korea(NRF) funded by the Ministry of Education, Science and Technology (2014R1A1A2056084). The fifth author was supported by Basic Science Research Program through the National Research Foundation of Korea(NRF) funded by the Ministry of Education, Science and Technology (NRF-2016R1D1A1B03934771).}


\date{}


\dedicatory{}

\begin{document}
	
\begin{abstract} The main aim of this paper is to prove a Bishop-Phelps-Bollob\'as type theorem on the unital uniform algebra $\mathcal{A}_{w^*u}(B_{X^*})$ consisting of all $w^*$-uniformly continuous functions on the closed unit ball $B_{X^*}$ which are holomorphic on the interior of $B_{X^*}$. We show that this result holds for $\mathcal{A}_{w^*u}(B_{X^*})$ if $X^*$ is uniformly convex or $X^*$ is the uniformly complex convex dual space of an order continuous absolute normed space. The vector-valued case is also studied.
\end{abstract}

 \maketitle

In 1961, Bishop and Phelps proved that the set of norm attaining functionals is dense in the dual space \cite{BP}. They questioned whether the same result holds for bounded linear operators and the answer was given two years later by Lindenstrauss \cite{L} who gave a counterexample proving that in general the answer is false. However, he also presented conditions on a Banach space to get positive results. On the other hand, Bollob\'as proved a stronger version for linear functionals which is nowdays known as the Bishop-Phelps-Bollob\'as theorem \cite{Bol}. This says that functionals and points which they almost attain their norms can be simultaneously approximated by norm attaining functionals and points which they
attain their norms. We highlight this theorem because it is the main motivation for the present work. If $X$ is a Banach space, $\e \in (0, 2)$ and $(x, x^*) \in B_X \times B_{X^*}$ satisfy the following inequality
\begin{equation*}
\re x^*(x) > 1 - \frac{\e^2}{2},
\end{equation*}
then there is $(y, y^*) \in S_X \times S_{X^*}$ such that
\begin{equation*}
|y^*(y)| = 1, \ \ \|y - x\| < \e \ \ \mbox{and} \ \ \|y^* - x^*\| < \e
\end{equation*}
(this version can be found in \cite[Corollary 2.4]{CKMM}).

Since 2008, with the seminal paper by Acosta, Aron, Garc\'ia and Maestre \cite{AAGM},
a lot of attention had been paid in the attempt to get Bishop-Phelps-Bollob\'as type theorems for bounded linear operators, homogeneous polynomials and multilinear mappings by putting conditions on the Banach spaces as Lindenstrauss did. Our aim here is to get a Bishop-Phelps-Bollob\'as type theorem for holomorphic functions. In \cite{AAGM2} the authors showed that if $X$ is  a complex Banach space with Radon-Nikod\'ym property and if $\mathcal{A}_u(B_X)$ stands for the space of all
uniformly continuous functions on the closed unit ball $B_X$ which are holomorphic on the interior endowed with the supremum norm, then the set of all norm attaining elements is dense in $\mathcal{A}_u(B_X)$. This result was sharpened to the denseness of the set of all strong peak functions on $\mathcal{A}_u(B_X)$ \cite[Theorem 4.4]{CLS}.

\vspace{0.3cm}

Before we give our results, we present the necessary background. Throughout the paper we consider Banach spaces over the complex field. We will use the notation $X^*$, $S_X$ and $B_X$ for the dual space, the unit sphere and the closed unit ball of $X$, respectively. We denote by $\A$ the unital uniform algebra of all $w^*$-uniformly continuous functions from $B_{X^*}$ into $\C$ which are holomorphic on the interior of $B_{X^*}$ endowed with the supremum norm
\begin{equation*}
\|f\|_{\infty} = \sup \left\{ |f(x^*)|: x^*\in B_{X^*} \right\}.
\end{equation*}
It is known that $\mathcal{A}_{w*}(B_{X^*})$ coincides with   $\mathcal{A}_{w*u}(B_{X^*})$ \cite{ACG}. An element $x_0^* \in S_{X^*}$ is said to be a {\it strong peak point} for $\A$ if there is $f \in \A$ such that $\|f\|_{\infty} = |f(x_0^*)| = 1$ and for every $w^*$-neighborhood $W$ of $x_0^*$, we have
\begin{equation*}
\sup_{x^* \in B_{X^*} \setminus W} |f(x^*)| < 1.
\end{equation*}
In this case, $f$ is said to be a {\it strong peak function at} $x_0^*$. We denote by $\Gamma$ the set of all strong peak points for $\A$.

Let $K$ be a nonempty compact Hausdorff space and let $A$ be a uniform algebra. Given $t \in K$, the function $\delta_t: A \longrightarrow \C$ defined by $\delta_t(f) = f(t)$ is called the evaluation functional at $t$. A set $S \subset K$ is said to be a boundary for the uniform algebra $A$ if for every $f \in A$, there is $x \in S$ such that $|f(x)| = \|f\|_{\infty}$.   If $S = \{x^* \in A^*: \ \|x^*\| = x^*(\textbf{1}) = 1 \}$ and $\ext_{\R}(S)$ stands for the set of all real extreme points of $S$, then $\Gamma_0(A) = \{t \in K: \delta_t \in \ext_{\R}(S)\}$ is a boundary for $A$ which is called the {\it Choquet boundary} of $A$.

 Let $A$ be a unital uniform algebra on $K$.  It is well-known that the Choquet boundary for $A$ consists exactly of the strong peak points for $A$ \cite[Proposition 4.3.4]{Dales}. Recall that an element $x$ of the unit ball of a Banach space $X$ is said to be a complex extreme point if $\sup_{0 \leq \theta \leq 2 \pi} \|x + e^{i \theta} y\| > 1$ for all nonzero $y \in X$. We denote by $\ext_{\C} (B_{X})$ the set of all complex extreme points of $B_X$. Note that $\Gamma$ is the Choquet boundary for $\A$ and $\Gamma$ is contained in $\ext_{\C}(B_{X^*})$. In particular, it is observed that $\Gamma = \ext_{\C}(B_{X^*})$ for finite dimensional Banach spaces $X$ \cite[Proposition 1.1]{CHL}.

To get a version of the Bishop-Phelps-Bollob\'as theorem on the space $\A$, we need to consider stronger peak functions. A point $x_0^* \in S_{X^*}$ is said to be a {\it strong peak point for $\A$ with respect to the norm} if there is $f \in \A$ such that $\|f\|_{\infty} = |f(x_0^*)| = 1$ and for all $\delta > 0$,
\begin{equation*}
\sup_{y^* \in B_{X^*} \setminus B(x_0^*, \delta)} |f(y^*)| < 1.
\end{equation*}
Equivalently, it happens when $\|f\|_{\infty} =|f(x_0^*)|=1$ and $\{x_n^*\} \stackrel{\| \cdot\|}{\longrightarrow} x_0^*$ whenever $\{x^*_n\} \subset B_{X^*}$ satisfies $\lim_{n\to \infty} |f(x_n^*)|=1$. The function $f$ is called a {\it strong peak function at $x_0^*$ with respect to the norm}. We denote by $\Gamma_s$ the set of all strong peak points for $\A$ with respect to the norm.
We note that for every $w^*$-neighborhood $W$ of $x_0^*$, there exists $\delta > 0$ such that $B(x_0^*, \delta) \subset W$. This implies that $\Gamma_s \subset \Gamma$.


\vspace{0.3cm}

In what follows, we present a non-linear version of the Bishop-Phelps-Bollob\'as theorem for the unital uniform algebra $\A$ by assuming that the set of all strong peak points with respect to the norm is norm dense in the unit sphere of the dual space. After that, we give some consequences like a Lindenstrauss-Bollob\'as type theorem as well as a vector-valued extension of the main result. Finally, we present some examples when $\Gamma_s$ is norm dense in $S_{X^*}$.

\vspace{0.3cm}

We start with the Bishop-Phelps-Bollob\'as theorem for $\A$.

\begin{theorem} \label{BPBpforAw} Let $X$ be a complex Banach space and suppose that $\Gamma_s$ is norm dense in $S_{X^*}$. Then, given $\e \in (0, 1)$, there exists $\eta(\e) > 0$ such that whenever $f \in \A$ with $\|f\|_{\infty} = 1$ and $x_0^* \in S_{X^*}$ satisfy
\begin{equation*}	
|f(x_0^*)| > 1 - \eta(\e),
\end{equation*}
there are $g \in \A$ with $\|g\|_{\infty} = 1$ and $x_1^*\in S_{X^*}$ such that
\begin{equation*}
|g(x_1^*)| = 1, \ \ \  \|g - f\|_{\infty} < \e \ \ \ \mbox{and} \ \ \  \|x_1^* - x_0^* \| < \e.
\end{equation*}
\end{theorem}

\noindent The following Urysohn type lemma plays an important role in the proof of the theorem.

\begin{lemma}\cite[Lemma 2.7]{CGK} \label{Urysohn} Let $A \subset C(K)$ be a unital uniform algebra and $\Gamma_0$ its Choquet boundary. Then, for any open subset $U$ of $K$ with $U \cap \Gamma_0 \ne\emptyset$ and for $0<\eps<1$, there exist $f\in A$ and $t_0\in U\cap\Gamma_0$ satisfying
\begin{itemize}
\item[(i)] $f(t_0) = \|f\|_{\infty} = 1$, \item[(ii)] $|f(t)|<\eps$ for every $t \in K \setminus U$ and
\item[(iii)] $|f(t)| + (1 - \e)|1 - f(t)| \leq 1$ for every $t \in K$.
\end{itemize}
\end{lemma}
\noindent Now we are ready to prove Theorem~\ref{BPBpforAw}.

\begin{proof}[Proof of Theorem~\ref{BPBpforAw}] Let $\e \in (0, 1)$ be given and define $\eta(\e) :=\frac{\e}{4} > 0$. Let $f \in \A$ with $\|f\|_{\infty} = 1$ and $x_0^* \in S_{X^*}$ be such that
\begin{equation*}
|f(x_0^*)| > 1 - \eta(\e).
\end{equation*}
Note that
\begin{equation*}
U_1=\left\{ x^*\in B_{X^*} :  \left|\frac{f(x_0^*)}{|f(x_0^*)|}-f(x^*)\right| < \eta(\e)\right\}
\end{equation*}
is a nonempty $w^*$-open set on $B_{X^*}$. Since $\Gamma_s$ is norm dense in $S_{X^*}$, there is $z_0^* \in \Gamma_s$ such that
\begin{equation*}
z_0^*\in U_1 \ \ \ \mbox{and} \ \ \ \|z_0^* - x_0^*\| < \frac{\e}{2}.
\end{equation*}
Let $h \in \A$ be a strong peak function at $z_0^*$ with respect to the norm. Then $ \|h\|_{\infty} = |h(z_0^*)| = 1$ and
\begin{equation} \label{BPBp1}
r := \sup_{y^* \in B_{X^*} \setminus B \left(z_0^*,\frac{\e}{2}\right)} |h(y^*)| < 1.
\end{equation}
Now define
\begin{equation*}
U_2 := U_1 \cap \{ x^* \in B_{X^*} : |h(x^*)| > r\}
\end{equation*}
which is a $w^*$-open set on $B_{X^*}$. Note that $\Gamma_s$ is contained in the Choquet boundary of $\A$. Since $z_0^* \in U_2 \cap \Gamma_s$, we can apply Lemma~\ref{Urysohn} to get $\phi \in \A$ and $x_1^* \in U_2$ satisfying the following conditions:
\begin{itemize}
\item[(i)] $\phi(x_1^*) = \|\phi\|_{\infty} = 1$,
\item[(ii)] $|\phi(x^*)| < \eta(\e)$ for all $x^* \in B_{X^*} \setminus U_2$ and
\item[(iii)] $|\phi(x^*)|+(1-\eta(\e))|1-\phi(x^*)|\le 1$ for all $x^*\in B_{X^*}$.
\end{itemize}
Now let, for $x^*\in B_{X^*}$,
\begin{equation*}
g(x^*) :=  \frac{f(x_0^*)}{|f(x_0^*)|} \phi(x^*) + (1 - \eta(\e))(1 - \phi(x^*))f(x^*).
\end{equation*}
Then $g\in \A$ and, by (i), we have $|g(x_1^*)| = \phi(x_1^*) = 1$.
By using (iii), for all $x^*\in B_{X^*}$, we get that
\begin{equation*}
|g(x^*)|\le |\phi(x^*)| +(1-\eta(\e))|1-\phi(x^*)|\le 1.
\end{equation*}
This shows that $\|g\|_{\infty} = |g(x_1^*)| = 1$. Since $x_1^* \in U_2$, we have that $|h(x_1^*)| > r$ and using (\ref{BPBp1}), $\|x_1^* - z_0^*\| < \frac{\e}{2}$. This implies that
\begin{equation*}
\|x_1^* - x_0^*\| \leq \|x_1^* - z_0^*\| + \|z_0^* - x_0^*\| < \e.
\end{equation*}
Finally we show that $\|g - f\|_{\infty} < \e$.
Indeed, we write
\begin{equation*}
g - f =  \left(\frac{f(x_0^*)}{|f(x_0^*)|} - f\right) \phi  - \eta(\e)(1 - \phi)f.
\end{equation*}
If $x^*\in U_2$, then
\begin{equation*}
\|g(x^*) - f(x^*)\| < \|\phi\|_{\infty} \eta(\e)+ \eta(\e) ( 1 + \|\phi\|_{\infty}) \|f\|_{\infty} = 3\eta(\e) < \e.
\end{equation*}
If $x^*\in B_{X^*}\setminus U_2$, then we use (ii) to get
\begin{equation*}
\|g(x^*) - f(x^*)\| \leq 2|\phi(x^*)| + \eta(\e)(1 + |\phi(x^*)|)\|f\|_{\infty} \\
<  4 \eta(\e) = \e.
\end{equation*}
Therefore, we have $\|g - f\|_{\infty} < \e$ and this finishes the proof.
\end{proof}

In \cite[Theorem B]{CM}, the authors proved that if $X$ is a Banach space whose dual is separable and has the approximation property, then the set of analytic functions whose Aron-Berner extensions attain their norm is dense in $\mathcal{A}_u(B_X)$. On the other hand, \cite[Proposition 4.4]{CLM} gives us a counterexample of a space which does not satisfy a Lindenstrauss-Bollob\'as type theorem for multilinear forms or multilinear mappings. Here, since $\mathcal{A}_{wu}(B_X)$ is isometrically isomorphic to $\mathcal{A}_{w^*u}(B_{X^{**}})$ (see \cite[Theorem 6.3]{ACG}), we have the following consequence of Theorem \ref{BPBpforAw}.

\begin{prop} Let $X$ be a complex Banach space and let $\Gamma_s(B_{X^{**}})$ be the set of all strong peak points for $A_{w^*u}(B_{X^{**}})$ with respect to the norm. Suppose that $\Gamma_s(B_{X^{**}})$ is dense in $S_{X^{**}}$. Then, given $\e > 0$, there exists $\eta(\e) > 0$ such that whenever $f \in \mathcal{A}_{wu}(B_X)$ with $\|f\|_{\infty} = 1$ and $x_0 \in S_X$ satisfy
\begin{equation*}
|f(x_0)| > 1 - \eta(\e),
\end{equation*}
there are $g \in S_{\mathcal{A}_{wu}(B_X)}$ and $z_0^{**} \in S_{X^{**}}$ such that
\begin{equation*}
|AB{g}(z_0^{**})| = 1, \ \ \ \|g - f\|_{\infty} < \e \ \ \ \mbox{and} \ \ \ \|z_0^{**} - x_0\| < \e,
\end{equation*}
where $ABg\in {\mathcal{A}_{w^*u}(B_{X^{**}})}$ stands for the Aron-Berner extension of $g$ to  $B_{X^{**}}$.
\end{prop}

Let $Y$ be a complex Banach space. We denote by $\mathcal{A}_{w^*u}(B_{X^*}, Y)$ the space of all $Y$-valued holomorphic functions on the interior of $B_{X^*}$ which are $w^*$-to-norm uniformly continuous on $B_{X^*}$ equipped with the supremum norm.
Theorem~\ref{BPBpforAw} can be extended to the vector-valued case as follows by using the same proof.

\begin{theorem} \label{BPBpforAw2} Let $X$ and $Y$ be complex Banach spaces and suppose that $\Gamma_s$ is norm dense in $S_{X^*}$. Then, given $\e \in (0, 1)$, there exists $\eta(\e) > 0$ such that whenever $f \in \mathcal{A}_{w^*u}(B_{X^*}, Y)$ with $\|f\|_{\infty} = 1$ and $x_0^* \in S_{X^*}$ satisfy
\begin{equation*}	
\|f(x_0^*)\|_Y > 1 - \eta(\e),
\end{equation*}
there are $g \in \mathcal{A}_{w^*u}(B_{X^*}, Y)$ with $\|g\|_{\infty} = 1$ and $x_1^*\in S_{X^*}$ such that
\begin{equation*}
\|g(x_1^*)\|_Y = 1, \ \ \  \|g - f\|_{\infty} < \e \ \ \ \mbox{and} \ \ \  \|x_1^* - x_0^* \| < \e.
\end{equation*}
\end{theorem}

We observe that there are Banach spaces $Y$ such that the Bishop-Phelps-Bollob\'as theorem does not hold for bounded linear operators from $\ell_1$ into $Y$ (see \cite[Remark 2.4]{AAGM}). Moreover, in \cite[Corollary 3.3]{ACKLM}, it was shown that such theorem for operators from $\ell_1^2$ into some Banach space $Y$ is also not true where $\ell_p^n$ is $\mathbb{R}^n$ with $\ell_p$ norm for $n\in \mathbb{N}$. Nevertheless, we will prove that if $X^* = \ell_1^n$ or $\ell_1$, then Theorem~\ref{BPBpforAw2} holds for all complex Banach spaces $Y$ (see Proposition \ref{BPBpforAw3} below). It is worth to mention that in Bishop-Phelps-Bollob\'as type theorem we consider norm-to-norm continuous operators and here we are considering $w^*$-to-norm continuous functions. These two continuities coincide when the domain is finite dimensional.

Next we present some spaces satisfying the condition that  $\Gamma_s$ is norm dense in $S_{X^*}$.
Let $X$ be a subspace of  $\mathbb{C}^J$ for a set $J$. If $x=(x(j))_{j\in J}$ is an element of $\mathbb{C}^J$, the absolute value $|x|$ is defined to be $|x|=(|x(j)|)_{j \in J}$. Let $x=(x(j))_{j \in J}$ and $y=(y(j))_{j \in J}$. We say that $x\le y$ whenever $x, y \in \mathbb{R}^J$ and $x(j) \le y(j)$ for all $j$. A norm $\| \cdot \|$ on $X$ is said to be an absolute norm if $(X, \|\cdot\|)$ is a Banach space and, if $x \in X$ and $|y| \leq |x|$ for some $y \in \mathbb{C}^J$ we have that $y\in X$ and $\|y\| \leq \|x\|$. In this case, we call $X$ as an absolute normed space. We denote by $e_j$ the $j$-th standard unit vector defined by $e_j(i) =0$ if $i\neq j$ and $e_j(j)=1$ for all $j$. We also assume that absolute normed spaces contain each $e_j$ and $\|e_j\|=1$. Notice that absolute normed spaces are complex Banach lattices they can be viewed as K\"{o}the spaces on the measure space $(J,2^J,\nu)$ where $\nu$ is the counting measure on a set $J$. We say that a Banach lattice $X$ is order continuous if every downward directed set $\{x_\alpha\}$  in $X$ with $\bigwedge_\alpha x_\alpha=0$, $\lim_\alpha \|x_\alpha\|=0$. For the reference about the order continuity and K\"othe spaces, see \cite{LT}.

It is well-known that an absolute normed space $X$ is order continuous if and only if $X$ is the closed linear span of the set $\{ e_j : j\in J\}$ of all standard unit vector basis (for the reference of this fact, see the proof of Proposition~7.1 in \cite{KLMM}.) If $X$ is order continuous, then $X^*$ is also an absolute normed space \cite[P. 29]{LT}. In fact, the dual space $X^*$ can be identified with the K\"othe dual $X'$ consisting of all $x^*=(x^*(j))_{j \in J}$ in $\mathbb{C}^J$ such that
\[ \|x^*\| = \sup\left\{ \left|\sum_{j\in J} x^*(j)x(j)  \right|: x=(x(j))_{j \in J} \in S_X\right\}<\infty.\]
The standard unit vector in $X^*$ will be denoted by $e_j^*$.

Given a complex Banach space $X$, the {\it modulus of complex convexity} is defined by
\[ H(\eps) = \inf \left\{ \sup_{0\le \theta\le 2\pi}\|x+e^{i\theta} y \| -1: x \in S_X, \|y\|\ge \eps \right\}.\]
A Banach space is said to be {\it uniformly complex convex} whenever $H(\eps)>0$ for all $\eps>0$.  The complex convexity and monotonicity is closely related in Banach lattices. Recall that a Banach space is said to be {\it uniformly monotone} if, for all $\eps>0$, we have
\begin{equation*}
M(\eps) = \inf \left\{ \||x| + |y|\| -1: \ x \in S_{X}, \ \|y\|\geq \eps \right\} > 0.
\end{equation*}
It is shown in \cite[Theorem 3.5]{Lee} that a complex Banach lattice (an absolute normed space) is uniformly complex convex if and only if it is uniformly monotone. It is also shown \cite[Proposition 2.2]{Lee} that if a complex Banach lattice is uniformly monotone, then it is order continuous. We will use these facts in the next result. It is known that $\ell_p$ and $\ell_p^n$ are uniformly monotone (and so is uniformly complex convex) for $1\le p<\infty$. The uniform complex convexity of Orlicz-Lorentz spaces is characterized in \cite{CKL}.

\begin{prop} \label{BPBpforAw3}
Suppose that $X$ is an order continuous subspace of $\mathbb{C}^J$ with an absolute norm and suppose that $X^*$ is uniformly complex convex. Let $F$ be a finite dimensional subspace spanned by a finite number of standard unit vectors of $X^*$. Then $S_F$ is contained in $\Gamma_s$ and $\Gamma_s$ is norm dense in $S_{X^*}$.
\end{prop}

\begin{proof} Suppose that $X^*$ is uniformly complex convex. Then it is uniformly monotone. By the above discussion, $X^*$ is order continuous. So $X^*$ is the norm closure of the span of all standard unit vectors $e_j^*$ in $X^*$. Thus it is enough to show that $S_F$ is contained in $\Gamma_s$.

Suppose that $F$ is the linear span of $\{e_j^*: j\in J_0\}$, where $J_0$ is a finite subset of $J$. Let $P:X^* \longrightarrow F$ be the natural projection defined by $P(x^*) = \sum_{j\in J_0} x^*(e_j)e_j^*$. Note that $P$ is $w^*$-to-norm continuous and $\|x^*\| = \| |Px^*| + |(I-P)x^*| \|$ for all $x^*\in X^*$.
Let $x_0^* \in S_{F}$. Since $X^*$ is uniformly complex convex, $S_F = \ext_{\C}(B_F)$ and  this is exactly the set of all strong peak points for $A_u(B_F)$  because $F$ is finite dimensional \cite[Proposition 1.1]{CHL}.
So there is a strong peak function $g$ at $x_0^*$  in $\mathcal{A}_u(B_{F})$.

We will show that the function $f: B_{X^*} \longrightarrow \K$ defined by
\begin{equation*}
f(x^*) := (g \circ P)(x^*) \ \ \ (x^* \in B_{X^*})
\end{equation*}
is a strong peak function for $\A$ at $x_0^*$ with respect to the norm and the proof will be done. Note first that $f\in \A$. To show that it is a strong peak function at $x_0^*$ with respect to the norm, assume that there are $\delta_0 > 0$ and a sequence $\{x_k^*\} \subset B_{X^*}$ such that $\lim_{k \rightarrow \infty} |f(x_k^*)| = 1$ but $\|x_k^* - x_0^*\| \geq \delta_0$ for every $k \in \N$. Since $\lim_{k \rightarrow \infty} |g(P(x_k^*))| = 1$ and $g$ is a strong peak function at $x_0^*$ in $A_u(B_F)$,  we get that $\lim_{k\rightarrow \infty}\|Px_k^* -x_0^*\|=0$.
Since $\|(Px_k^*-x_0^*) + (I-P)(x_k^*)\| = \|x_k^* - x_0^* \| \geq \delta_0$, we may assume that $\|(I-P)(x_k^*)\|\ge \delta_0/2$ for all $k$. Nevertheless, for all $k \in \N$, we get that
\begin{eqnarray*}
1 \ge \|x^*_k\| &=& \|P(x_k^*) + (I-P)(x_k^*)\| \\
&\geq& \|P(x_0^*) + (I-P)(x_k^*)\|-\|x_0^* - Px_k^*\| \\
&=& \||P(x_0^*)| + |(I-P)(x_k^*)|\|-\|x_0^* - Px_k^*\| \\
&\geq& 1+ M(\delta_0/2) -\|x_0^* - Px_k^*\|.
\end{eqnarray*}
This shows that $M(\delta_0/2)\le 0$ which is a contradiction since $X^*$ is uniformly monotone.
\end{proof}

Recall that $x^*_0 \in B_{X^*}$ is said to be a $w^*${\it -strongly exposed point} if there exists $x_0 \in S_{X}$ such that $x^*_0(x_0) = 1$ and $\lim_{\delta \rightarrow 0+}{\rm diam}( S(x_0, \delta) ) = 0$,
where $S(x_0, \delta) = \left\{ x^*\in B_{X^*} : {\rm Re} \ x^*(x_0)>1-\delta\right\}$. We have the following proposition.

\begin{prop} \label{BPBpforAw4} Let $X$ be a complex Banach space. If $x_0^*\in B_{X^*}$ is a $w^*$-strongly exposed point, then $x_0^*$ is a strong peak point for $\A$ with respect to the norm. In particular, if $X$ is reflexive and $X^*$ is locally uniformly convex, then $\Gamma_s = S_{X^*}$.
\end{prop}

\begin{proof} Suppose that $x_0^* \in S_{X^*}$ is a $w^*$-strongly exposed point. Then there exists $x_0 \in S_X$ such that $\re x_0^*(x_0) = 1$ and for every $\{x_n^*\} \subset B_{X^*}$ with $\re x_n^*(x_0) \longrightarrow 1$, we have that $\|x_n^* - x_0^*\| \longrightarrow 0$ when $n \rightarrow \infty$.

  Define $f: B_{X^*}	\longrightarrow \C$ by
	\begin{equation}
	f(x^*) := \frac{1 + x^*(x_0)}{2} \ \ \ (x^* \in B_{X^*}).	
	\end{equation}
Then $f \in \A$, $\|f\|_{\infty} \leq 1$ and
\begin{equation*}
\|f\|_{\infty} \geq \frac{1 + \re x_0^*(x_0)}{2} = 1.	
\end{equation*}	

Hence, $\|f\|_{\infty} = f(x_0^*) = 1$. Let $\{x_n^*\} \subset B_{X^*}$ satisfy that $|f(x_n^*)| \longrightarrow 1$ as $n \rightarrow \infty$ which means $\left| \frac{1 + x_n^*(x_0)}{2} \right| \longrightarrow 1$. Since
\begin{equation*}
|1 + x_n^*(x_0)|^2 + |1 - x_n^*(x_0)|^2 = 2 (1 + |x_n^*(x_0)|^2),
\end{equation*}
we have that
\begin{equation*}
|1 - x_n^*(x_0)|^2 \leq 4 - |1 + x_n^*(x_0)|^2 \longrightarrow 0 \ \ \ (n \rightarrow \infty).
\end{equation*}

 This implies that $x_n^*(x_0) \longrightarrow 1$ and so $\re x_n^*(x_0) \longrightarrow 1$. Hence $\|x_n^* - x_0^*\| \longrightarrow 0$ when $n \rightarrow \infty$. This shows that $x_0^*$ is a strong peak point for $\A$ with respect to the norm.

Now suppose that $X$ is reflexive and $X^*$ is locally uniformly convex. We will prove that every point $x_0^* \in S_{X^*}$ is a $w^*$-strongly exposed point. Indeed, since $X$ is reflexive, there exists $x_0 \in S_X$ such that $\re x_0^*(x_0) = 1$. Let $\{x_n^*\} \subset B_{X^*}$ be such that $\re x_n^*(x_0) \longrightarrow 1$ when $n \rightarrow \infty$. Then
\begin{equation*}
1 \geq \left\| \frac{x_n^* + x_0^*}{2} \right\| \geq \re \left(  \frac{x_n^*(x_0) + x_0^*(x_0)}{2} \right) \stackrel{n \rightarrow \infty}{\longrightarrow} 1.	
\end{equation*}
Since $X^*$ locally uniformly convex, we get that $\|x_n^* - x_0^*\| \longrightarrow 0$ when $n \rightarrow \infty$. By the previous paragraph, $\Gamma_s = S_{X^*}$ as desired.
\end{proof}

\end{document}